\documentclass[runningheads]{llncs}
\usepackage{graphicx}
\usepackage[usenames,dvipsnames,condensed]{xcolor}
\usepackage{caption}
\usepackage{subcaption}
\usepackage{mathtools}
\usepackage[usenames,dvipsnames,condensed]{xcolor}

\usepackage{amsthm}
\usepackage{amsfonts}
\usepackage[UKenglish]{babel}
\usepackage[usenames]{xcolor}
\usepackage{graphicx}
\usepackage{soul}
\usepackage{stfloats}
\usepackage{morefloats}
\usepackage{cite}
\usepackage{lscape}
\usepackage{epstopdf}
\usepackage{mathrsfs}

\usepackage{amsmath,amstext,amsopn,amsfonts,eucal,amssymb}
\usepackage{graphicx,wrapfig,url}
\usepackage{color}

\def \mc {\mathcal}

\def \a {\alpha}
\def \bt {\beta}

\begin{document}
\title{ Algebraic Systems for DNA Origami Motivated from  Temperley-Lieb Algebras
}
\titlerunning{Algebraic Systems for DNA Origami}
%
\author{James Garrett, Nata\v sa Jonoska, Hwee Kim, Masahico Saito \\
	%
	%
	\institute{
	Department of Mathematics and Statistics, University of South Florida
	4202 E. Fowler Ave, Tampa, FL 33620, USA}
}
\maketitle              
\begin{abstract}
	We initiate an algebraic approach to study DNA origami structures by associating an element from a monoid to each structure. We identify two types of basic building blocks
	and describe an DNA origami structure with their composition. These building blocks are taken as generators of a monoid, called origami monoid,  and, motivated by the well studied Temperley-Lieb algebras, we identify a set of relations that characterize the origami monoid. We also present several observations about the Green's relations for the origami monoid and 
	study the relations to 
	a cross product of Jones monoids that is a morphic image of an origami monoid.\\[2mm]
\end{abstract}
\section{Introduction}\label{intro-sec}


In the past few decades, bottom-up assemblies at the nano scale have introduced new materials and molecular scaffoldings producing structures that have wide ranging applications (eg. \cite{Geim,Yaghi}),  even materials that seem to violate standard chemistry behavior (eg. \cite{Oganov}).
``DNA origami", introduced by Rothemund \cite{rothemund-folding-2006} in 2006, significantly facilitated the construction of $\sim100 \times 100 ${\it nm} 2D DNA nanostructures.  The method
typically involves combining an M13 single-stranded cyclic viral molecule called {\it scaffold}  with 200-250 short
{\it staple strands} to produce about  100{\it nm} diameter 2D shapes \cite{rothemund-folding-2006}, 
and more recently also to produce a variety of 3D constructs (e.g. \cite{D+09}).
Fig.~\ref{origami_clean} (left) shows a schematic of an origami structure,
where the black thick line outlines a portion of the cyclic vector plasmid outlining the shape, and the colored lines are schematics of the short strands that keep the cyclic molecule folded in the shape. 
Because the chemical construction of DNA origami is much easier than 
previous methods, this form of DNA nanotechnology has become popular, with perhaps 300 
laboratories in the world today focusing on it.

Although numerous laboratories around the world are successful in achieving various shapes with DNA origami,
theoretical understanding and characterizations of these shapes is still lacking.
With this paper we propose  an
algebraic system
to  describe and investigate DNA origami structures. The staple strands usually have $2$-$4$ segments of about 8 bases joining $2$-$3$ locations (folds) of the scaffold. All cross-overs between two staple strands and between two neighboring folds of the scaffold are antiparallel. We divide the DNA origami structure to local scaffold-staples interactions and to such local interactions we associate a generator of a monoid which we call an {\it origami monoid}.
The origami monoid we present here is closely related to the
Jones monoid~\cite{BDP,IdemP} which is in fact considered as the  quotient of the well studied Temperley-Lieb algebra~\cite{temperley}. We show that a DNA origami structure can be described as an element of an origami monoid (a word over the set of generators) and propose a set of rewriting rules that are plausible for DNA segments to conform in DNA origami.
The set of rewriting rules in some sense describe the equivalence classes of the DNA origami structures. The number of
generators of an origami monoid depends on the number of parallel folds of the scaffold in the DNA origami.
We show that a cross product of two Jones monoids is a surjective image of an origami monoid, and
we study the structure of the origami monoids  through the Greens relations.
We characterize the origami monoids for small number of scaffold folds  and propose several conjectures for the general origami monoids.


\begin{figure}[h!]
	\centering
	\begin{minipage}[h]{.4\textwidth}
		\centering
		\includegraphics[scale=.4]{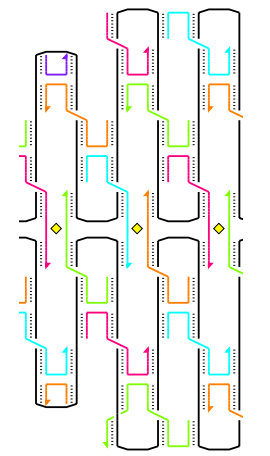}
	\end{minipage}%
	\begin{minipage}[h]{.48\textwidth}
		\centering
		\includegraphics[scale=.3]{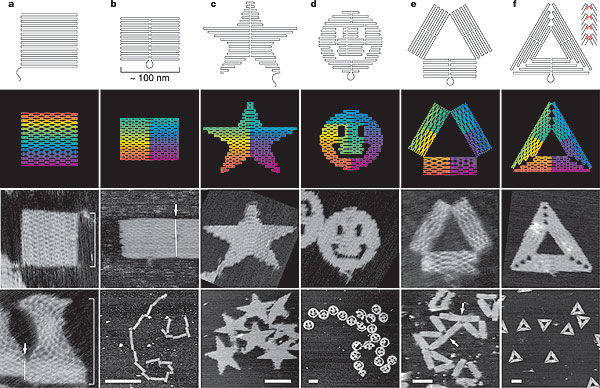}
		\label{DNAshapes}
	\end{minipage}
	\caption{(Left) A schematic figure of DNA origami structure with scaffold in black and staples in color,  (Right) Various shapes made by DNA origami, both from  \cite{rothemund-folding-2006}}
	\label{origami_clean}
\end{figure}

\vspace{-1cm}


\section{Preliminaries}\label{sec:prelim}

\subsection{Temperley-Lieb algebra and Jones monoids}

Temperley-Lieb algebras have been used in many fields,
particularly in physics and  knot theory (see, for example, \cite{temperley,BDP,Kauff,LF}). A generator of the algebra $h_i$  where there are $i-1$ strings to the left is depicted in Fig.~\ref{lieb1}(A) \cite{Kauff}.
 Multiplication of two elements corresponds to  concatenation of diagrams, placing the diagram of the first element on top of the second. The relations of the algebra follow the diagrams depicted in Fig.~\ref{lieb1}(B), (C) and (D) where $\delta$ is an element of a ring. Here we use the monoid versions of Temperley-Lieb algebras,  called Jones monoids 
\cite{BDP,LF}. 
The Jones monoid is obtained by taking $\delta=1$. 
Thus we consider, for each $n$,  the Jones monoid ${\mathcal J}_n$ generated by $h_i$,
$i=1, \ldots, n-1$, and relations
$$
(B) \,\, h_ih_jh_i = h_i\ \text{ for } \ |i-j|=1,\
(C) \,\, h_ih_i = h_i \ 
\
(D) \,\, h_ih_j =  h_jh_i \ \text{ for } \ |i-j|\geq1.
$$


\vskip -20pt

\begin{figure}[h]
	\centering
	\includegraphics[scale=0.45]{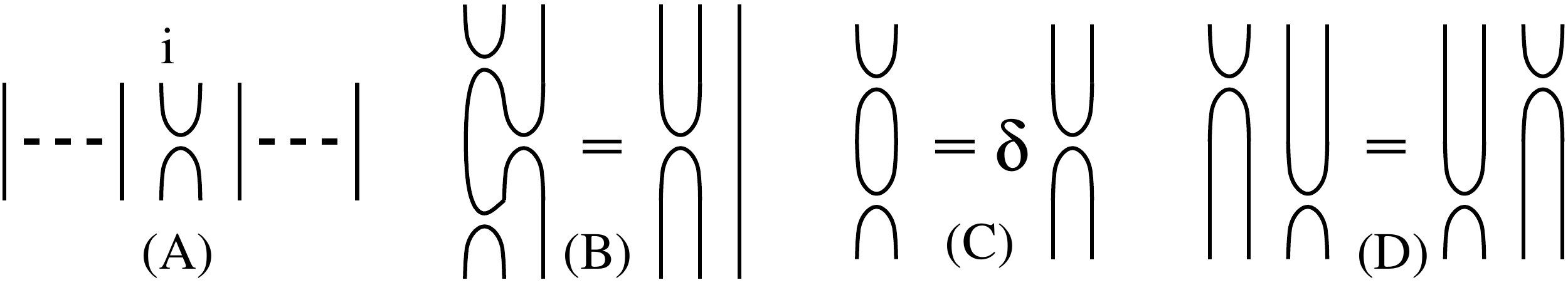}
	\caption{The generators (A) and relations (B,C,D) of the Temperley-Lieb algebra
	}
	\label{lieb1}
\end{figure}



\vspace{-1cm}

\subsection{String rewriting system}\label{sec:string}

An alphabet $\Sigma$ is a non-empty finite set of symbols. 
A word over $\Sigma$ is a finite sequence of elements (symbols) from $\Sigma$ and $\Sigma^*$ is the set of all words over $\Sigma$.
This set includes the empty string, the word containing no symbols, often written as 1. A word $u$ is called a {\it factor of a word $v$} if there exist words $x$ and $y$, which may be empty, such that $v = xuy$. Note that this is also sometimes referred to as a subword.

A string rewriting system, $(\Sigma,R)$ consists of an alphabet $\Sigma$ and a set of rewriting rules, $R$, which is a binary relation on $\Sigma^*$. An element $(x,y)$ of $R$ is called a rewriting rule, and is written $x \rightarrow y$.
We extend $R$ to factors of words 
$\xrightarrow[\text{R}]{}$, where for any $s,t\in \Sigma^*$, $s\xrightarrow[\text{R}]{}t$ if  there exist $x,y,u,v \in \Sigma^*$ such that $s = uxv$, $t=uyv$, and $x\rightarrow y$. 
We also write $s\rightarrow t$ for simplicity if no confusion arises.

If there is a sequence of words $u=x_1 \rightarrow x_2 \rightarrow \cdots \rightarrow x_n=v$ in a rewriting system
$(\Sigma^*, R)$, we write $u \rightarrow_* v$.
An element $x \in \Sigma^*$  is {\it confluent} if for all $y, z \in \Sigma^*$ such that
$x \rightarrow_* y$ and $x \rightarrow_* z$, there exists $w \in \Sigma^*$
such that $y \rightarrow_* w$ and $z \rightarrow_* w$. If all words in $\Sigma^*$ 
are confluent,
then $(\Sigma^*, R)$ is called {\it confluent}. In particular, if $R$ is symmetric, then the system $(\Sigma^*, R)$ is confluent.

\subsection{Monoids and  Green relations}


A monoid is a pair $(M,\cdot)$ where $M$ is a set and $\cdot$ is an 
associative binary operation on $M$ that has an identity element $1$. The set $\Sigma^*$
is a (free) monoid generated by  $\Sigma$ 
with word concatenation as the binary operation, and the empty string as the identity element.
Presentations of monoids are defined from the free monoid in a manner similar to presentations of groups.

For a monoid $M$, the {\it principal left (resp. right) ideal} generated by  $a \in M$  is defined by $Ma=\{ xa\ | \ x \in M\}$ (resp. $aM$), and the {\it principal two-sided ideal} s
$MaM$. Green's relations $\mathscr{L}$, $\mathscr{R}$, and $\mathscr{J}$ are defined for $a, b \in M$ by
$a \mathscr{L} b$ if $Ma=Mb$, $a \mathscr{R} b $ if $aM=bM$ and $a \mathscr{J} b$ if $MaM=MbM$.
Green's $\mathscr{H}$ relation is defined by $a \mathscr{H} b$ if $a \mathscr{L} b$ and $a \mathscr{R} b$. 
Green's $\mathscr{D}$ relation is defined by $a \mathscr{D} b$ if there is $c$ such that $a \mathscr{L} c $ and $c \mathscr{R} b$. The equivalence classes of $\mathscr{L}$ are called $\mathscr{L}$-classes, and similarly for the other relations. In a finite monoid, $\mathscr{D}$ and $\mathscr{J}$ coincide. 
The $\mathscr{D}$-classes are represented in a matrix form called {\it egg boxes}, where the rows represent $\mathscr{R}$-classes, columns $\mathscr{L}$-classes, 
and each entry is a box containing elements of $\mathscr{H}$-classes. 
See \cite{Pin} for more details.

\begin{example}\label{ex:J3}
	{\rm
		
		In \cite{IdemP}, $\mathscr{D}$-classes are obtained for Jones and related monoids. Here we include an example of a $\mathscr{D}$-class of ${\cal J}_3$, 
		which has a class consisting of the identity element and another class of $(2 \times 2)$-matrix below, such that each element is a box of an $\mathscr H$-class: 
		\[
		\left[
		\begin{array}{rr}
		h_1 & h_1 h_2 \\
		h_2 h_1 & h_2 
		\end{array}
		\right]
		\]
		where rows   $\{ h_1 , h_1 h_2 \} $, $\{ h_2 h_1, h_2 \}$, are the  $\mathscr{R}$-classes   
		and  columns  $\{ h_1 , h_2 h_1 \} $, $\{ h_1 h_2, h_2 \}$ are the $\mathscr{L}$-classes in this $\mathscr D$-class.  
		In particular, we see that multiplying $h_1$ and $h_1h_2$ with $h_i$ to the right 
		gives rise to the same right ideal. 
	}
\end{example}


\section{Origami monoid $\mc O_n$}\label{sec:monoid}

\subsection{Generators}\label{sec:gen}


Here we identify  simple building blocks in DNA origami structures. With each
block type we associate a generator of a monoid and derive string rewriting systems to describe DNA structures.
We have two motivations for our choices. (1) In Fig.~\ref{origami_clean}(left), one notices repeated patterns of 
simple building blocks whose concatenation builds  a larger structure.
One type of these patterns is a cross-over by the staple strands, and the other is a cross-over of the scaffold strand.
Thus, a natural approach to describe DNA origami 
structures symbolically is to associate generators  of an algebraic system to 
simple building blocks, and to take multiplication in the system to be presented as 
concatenation of the blocks. 
(2) In knot theory, a knot diagram is decomposed into basic building blocks of crossings or tangles. For the Kauffman bracket version of the Jones polynomial~\cite{Kauff}, for example, the Temperley-Lieb algebras, 
whose generators resemble building blocks observed in Fig.~\ref{origami_clean}(left), are used.
The Temperley-Lieb algebras have been extensively studied in physics and knot theory,
so that algebraic structures of their  variants  are also
of interest, besides their relations to  DNA origami.



For a positive integer $n$ we define a monoid ${\mathcal O}_n$,  
where $n$ represents the number of vertical double stranded DNA strands, that is,
$n$ is the number  of parallel folds of the scaffold. For the structure in Fig.~\ref{origami_clean}, $n=6$. The generators of ${\mathcal O}_n$ are denoted by $\alpha_i$ (anti-parallel staple strands cross-over)  and $\beta_i$ (antiparallel scaffold strand cross-over) for $i=1, \ldots, n-1$,
as depicted in Fig.~\ref{generators2}.
The subscript  $i$ represents the position of the left scaffold for $\alpha_i$ and $\beta_i$, respectively, by starting at 1 from the leftmost scaffold strand fold and counting 
right (Fig.~\ref{context}).

Because DNA is oriented $5'$-$3'$ and the strands in the double stranded DNA are oppositely oriented, we define an orientation within the generators. Because parallel scaffold strands are obtained by folding of the scaffold, 
consecutive scaffold strands run in alternating directions, while staple strands run in the opposite direction to the scaffold, and by default we take that the first scaffold runs in an upwards direction. In this way, the direction of the scaffold/staple strands for any particular $\alpha_i$ or $\beta_i$ depends entirely on the parity of $i$, as shown in Fig. \ref{generators2}.

\vspace{-.5cm}

\begin{figure}[h!]
    \begin{minipage}{.65\textwidth}
	\centering
	\begin{subfigure}[h]{.25\textwidth}
		\centering
		\includegraphics[scale=.2]{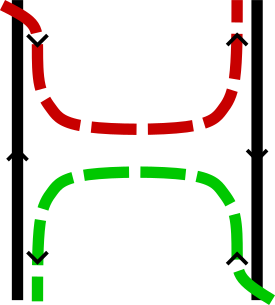}
		\caption{$\alpha_i$, $i$ odd}
		\label{alpha2}
	\end{subfigure}%
	\begin{subfigure}[h]{.25\textwidth}
		\centering
		\includegraphics[scale=.2]{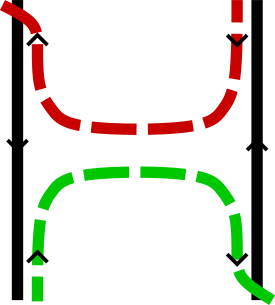}
		\caption{$\alpha_i$, $i$ even}
		\label{alpha3}
	\end{subfigure}%
	\begin{subfigure}[h]{.25\textwidth}
		\centering
		\includegraphics[scale=.2]{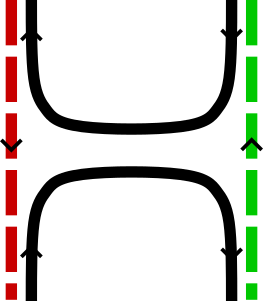}
		\caption{$\beta_i$, $i$ odd}
		\label{beta2}
	\end{subfigure}%
	\begin{subfigure}[h]{.25\textwidth}
		\centering
		\includegraphics[scale=.2]{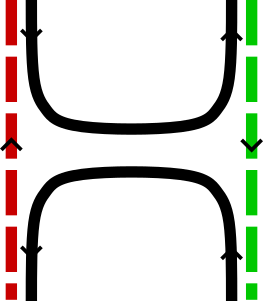}
		\caption{$\beta_i$, $i$ even}
		\label{beta3}
	\end{subfigure}
	\caption{The generators identified}
	\label{generators2}
	
	\end{minipage}%
	\begin{minipage}{.30\textwidth}
	    \centering
	\includegraphics[scale=.2]{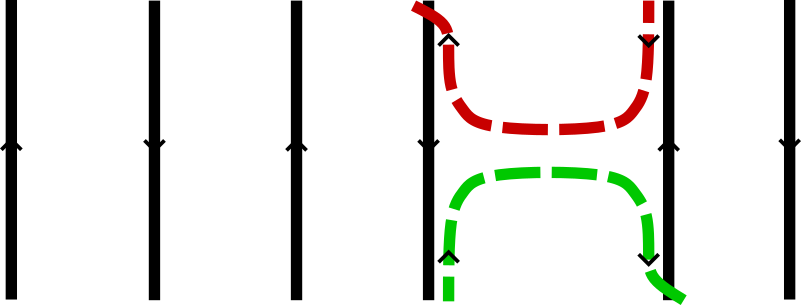}
	\captionsetup{justification=centering}
	\caption{$\alpha_4$ in the context of a 6-fold stranded structure.}
	\label{context}
	\end{minipage}\vspace{-.5cm}
\end{figure}

Fig. \ref{context} shows $\alpha_4$ as an example of the ``full picture'' of one of these generators. For the sake of brevity, 
we neglect to draw the extra scaffold and staple strands in most diagrams, but it may be helpful to imagine them when we describe their concatenation.
As in Fig.~\ref{context}, parallel scaffolds in generators do not have counterpart parallel staples.

\subsection{Concatenation as a monoid operation}\label{concat}

To justify modeling DNA origami structures by words in the generators we make  a correspondence between concatenations of generators $\alpha_i$, $\beta_i$ and concatenations of DNA segments. For a natural number $ n\geq2$, the set of generators of the monoid $\mc O_n$ is 
the set $\Sigma_{n}=\{\alpha_{1},\alpha_{2},\dots,\alpha_{n-1},\beta_{1},\beta_{2},\dots,\beta_{n-1}\}$.
For a product of two generators $x_i$ and $y_j$
in $\Sigma_n$,
we place the diagram of the  first generator above the second, lining up the scaffold strings of the two generators
and then, 
we connect each respective scaffold string. If the two generators are adjacent, that is, if for indices $i$ and $j$ 
it holds $|i-j|\leq1$, then we also connect their staples as described below. Otherwise, if $|i-j|\ge 2$, no staple connection is performed and the concatenation 
is finished.

We define a convention of connecting staples for adjacent generators, which is motivated from the manner in which
staples connect in Fig.~\ref{origami_clean}. Note how the staples of $\alpha$-type protrude ``outside" of the scaffold in Fig.~\ref{generators2}. We refer to these ends of a staple as an ``extending staple-ends'', and all other staple ends
as ``non-extending staple-ends''. We connect staples everywhere 
\textit{except} when two non-extending staple-ends would have to cross a scaffold to connect
(recall that the scaffold strands are connected first), as can be seen in Fig. \ref{alphabeta} and Fig. \ref{combinationgroup6}.

Our choice of coloring staples in the figure is arbitrary, and we re-color staples in the same color
if they get connected when we concatenate generators. 
By exhausting all possibilities, one can see that under our convention of connection, the staples remain short by concatenation without joining more than three scaffold folds.
Note that concatenation of three or more generators is associative because generators can be
connected in an associative manner 
following the rules described above.



\begin{figure}[h!]
	\centering
	\begin{minipage}[h]{.33\textwidth}
		\centering
		\includegraphics[scale=.13]{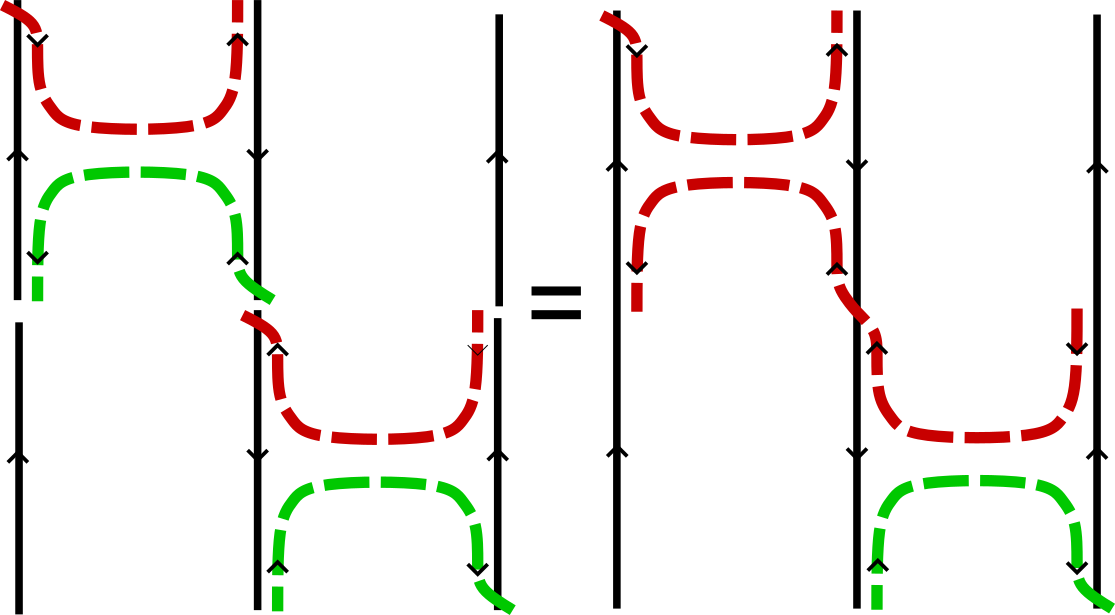}
		\caption{$\alpha_i\alpha_{i+1}$, $i$ odd}
		\label{concatexample1}
	\end{minipage}%
	\begin{minipage}[h]{.33\textwidth}
		\centering
		\includegraphics[scale=.13]{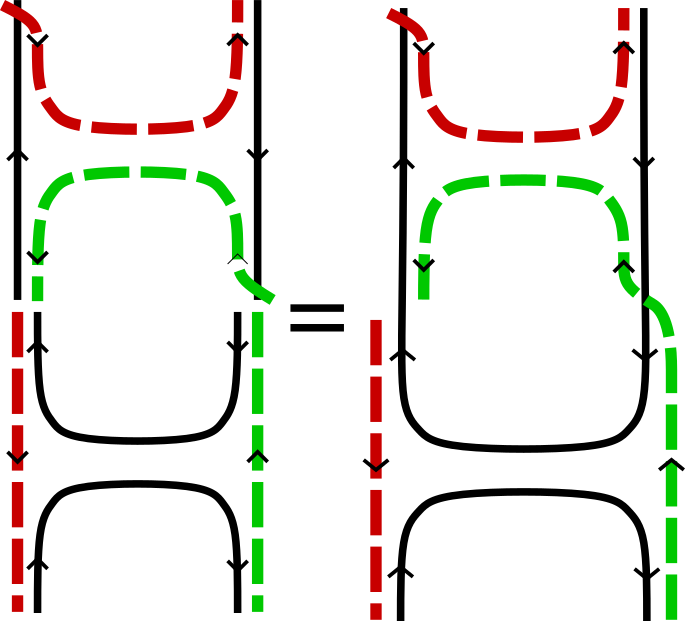}
		\caption{$\alpha_i\beta_i$, $i$ odd}
		\label{alphabeta}
	\end{minipage}
	\begin{minipage}[h]{.33\textwidth}
		\centering
		\includegraphics[scale=.13]{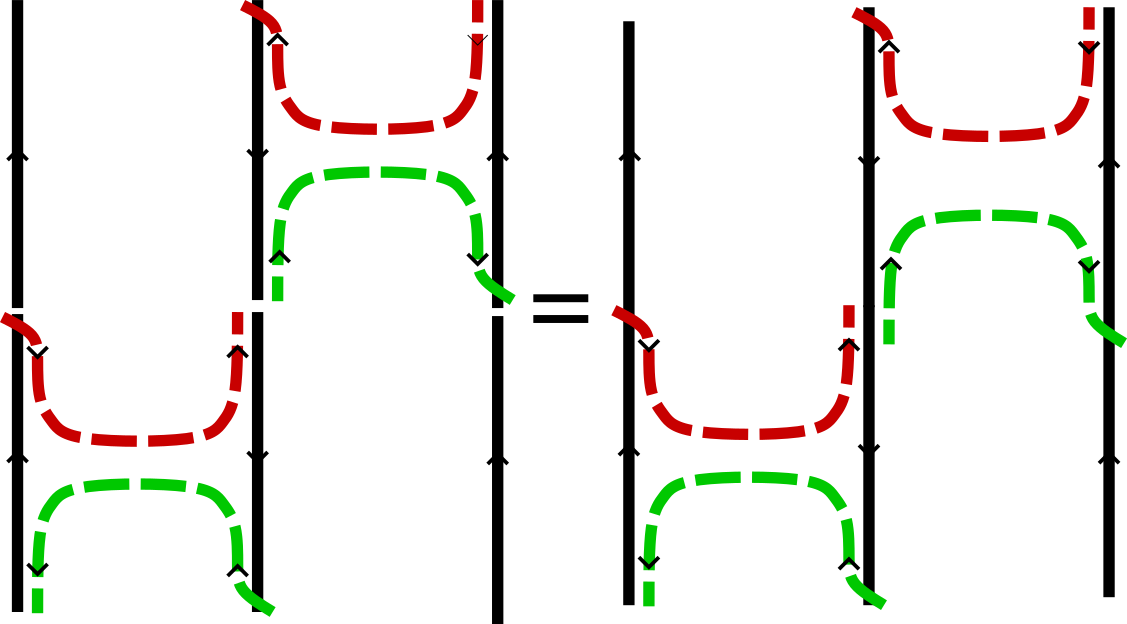}
		\caption{$\alpha_i\alpha_{i-1}$, $i$ odd}
		\label{combinationgroup6}
	\end{minipage}
\end{figure}


\subsection{Relations in $\mc O_n$}\label{sec:writing}


\begin{sloppypar}
	The rewriting rules (i.e., the relations within the monoids) 
	are motivated from similarity between the DNA origami structures as seen in Fig.~\ref{origami_clean}(left) and the diagrams of Temperley-Lieb algebras in Fig.~\ref{lieb1}.
	It is deemed that the relations of  Temperley-Lieb algebras simplify the DNA origami structure, and may be useful for designing efficient and more solid structures by the rewriting rules proposed below.
	The figures in this section are for justifying feasibility of corresponding DNA structures, and to represent rewriting system diagrammatically.
\end{sloppypar}

\subsubsection{Rewriting rules.}
For 
$\Sigma_{n}=\{\alpha_{1},\alpha_{2},\dots,\alpha_{n-1},\beta_{1},\beta_{2},\dots,\beta_{n-1}\}$, 
we establish a set of rewriting rules that allows us simplification of the DNA structure description.
Define a string rewriting system $(\Sigma_{n},R)$ as follows.

 To ease the notation, we define bar on $\Sigma_{n}$ by $\overline{\alpha_i}=\beta_i$ and $\overline{\beta_i}=\alpha_i$, and  extend this operation to the free monoid by defining $\overline{w}$ for a word $w$ by applying bar to each letter of $w$.
Let $\gamma \in \{\alpha,\beta\}$ and $i \in \{1,\dots,n-1\}$, then we have:

$$
\begin{array}{llrlll}
{\rm (1)}\  & {\rm (Idempotency)} & \gamma_{i}\gamma_{i} & \rightarrow &\gamma_{i} & \\
{\rm (2)}\  & {\rm (Left \ TL \ relation)} &\gamma_{i}\gamma_{i+1}\gamma_{i} & \rightarrow & \gamma_{i} & \\
{\rm (3)}\  & {\rm (Right\  TL \ relation)} &\gamma_{i}\gamma_{i-1}\gamma_{i} & \rightarrow & \gamma_{i} & \\
{\rm (4)}\  & {\rm (Inter-commutation)} &\gamma_{i}\overline{\gamma_{j}} & \rightarrow & \overline{\gamma_{j}} \gamma_{i}, \ {\rm for} \  |i-j|\geq 1  \\
{\rm (5)}\  & {\rm (Intra-commutation)} &\gamma_{i}\gamma_{j} &\rightarrow & \gamma_{j}\gamma_{i}, \ {\rm for} \  |i-j|\geq2
\end{array}
$$

\noindent
The rules are extended to $\Sigma_{n}^*$ as described in Section~\ref{sec:string}.

The rewriting rules are 
inspired by Temperley-Lieb algebras, and they are also reflected in the reality of the diagrams of DNA origami, as shown in Figs.~\ref{rule1} and \ref{rules}.
Specifically, a pattern in the left of Fig.~\ref{rule1} (a) has a small staple circle, which is deemed to be simplified by the right side. Staple strands are holding the scaffold in a certain position (obtained to the right of the arrow) and the  cyclic staple only reinforces the structure. The small circle of a scaffold  in Fig.~\ref{rule1} (b) left cannot form
in DNA origami, and therefore is 
simplified to the structure on the  right of the arrow. 


\begin{figure}[h!]
	\centering
	\begin{subfigure}[h]{.45\textwidth}
		\centering
		\includegraphics[scale=.12]{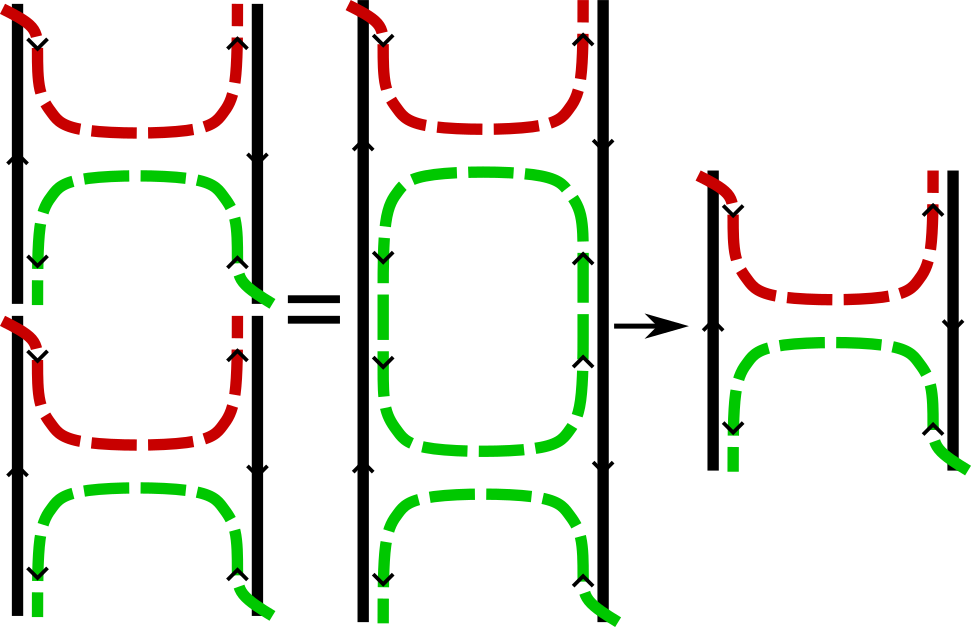}
		\caption{$\alpha_i\alpha_i$, $i$ odd}
		\label{combination1}
	\end{subfigure}%
	\begin{subfigure}[h]{.45\textwidth}
		\centering
		\includegraphics[scale=.12]{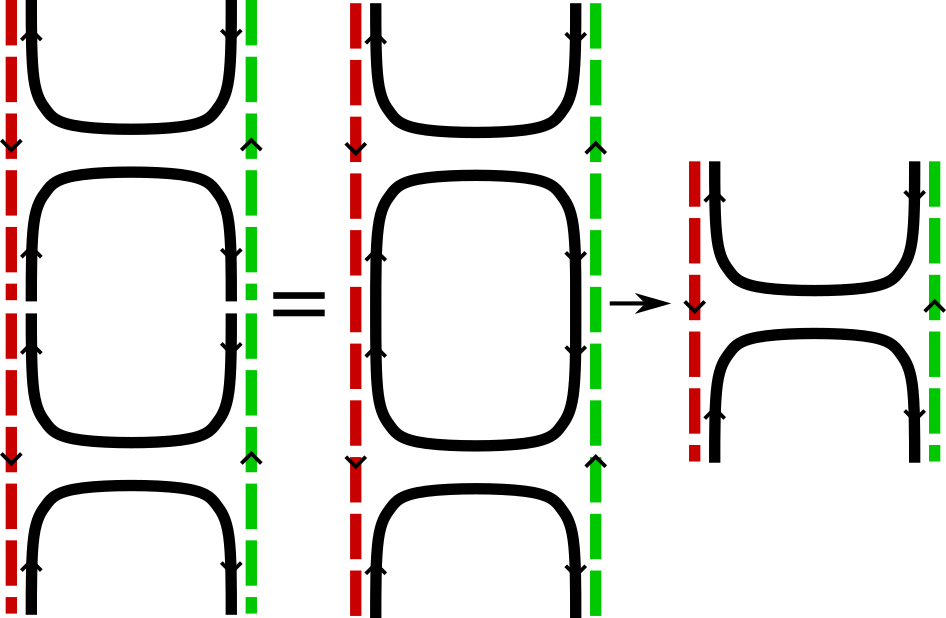}
		\caption{$\beta_i\beta_i$, $i$ odd}
		\label{combination3}
	\end{subfigure}
	\caption{Two examples of idempotency}
	\label{rule1}
\end{figure}


\begin{figure}[h!]
	\centering
	\begin{minipage}[h]{.21\textwidth}
		(A) \hskip 1mm\includegraphics[scale=.08]{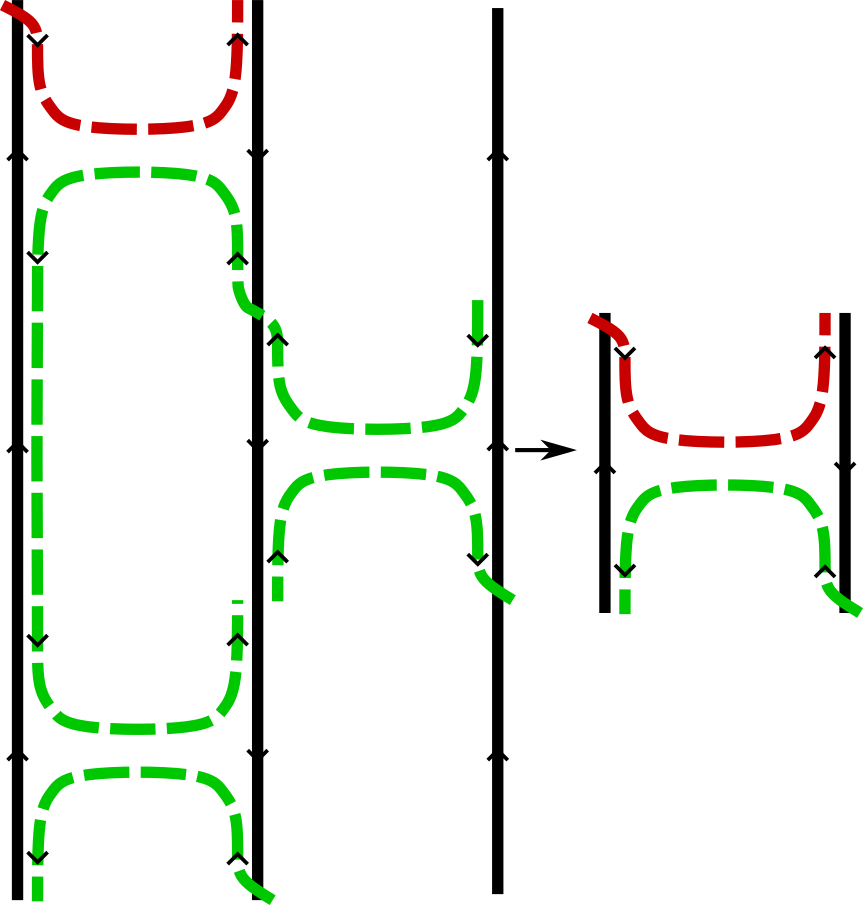}
		\label{rule2}
	\end{minipage}%
	\hskip 3mm
	\begin{minipage}[h]{.43\textwidth}
		(B) \hskip 1mm 
		\includegraphics[scale=.1]{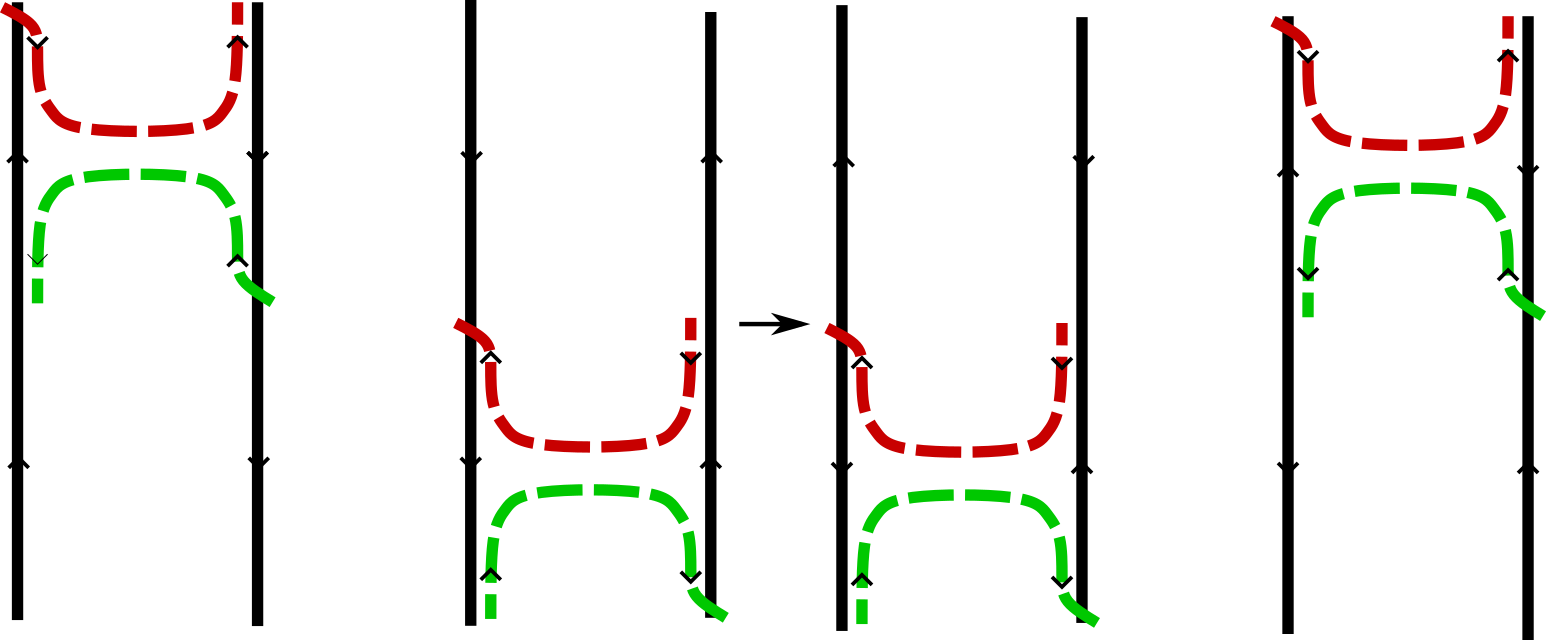}
		\label{rule4}
		\hskip 1mm
	\end{minipage}%
	\begin{minipage}[h]{.33\textwidth}
		(C)\hskip 2mm\centering
		\includegraphics[scale=.1]{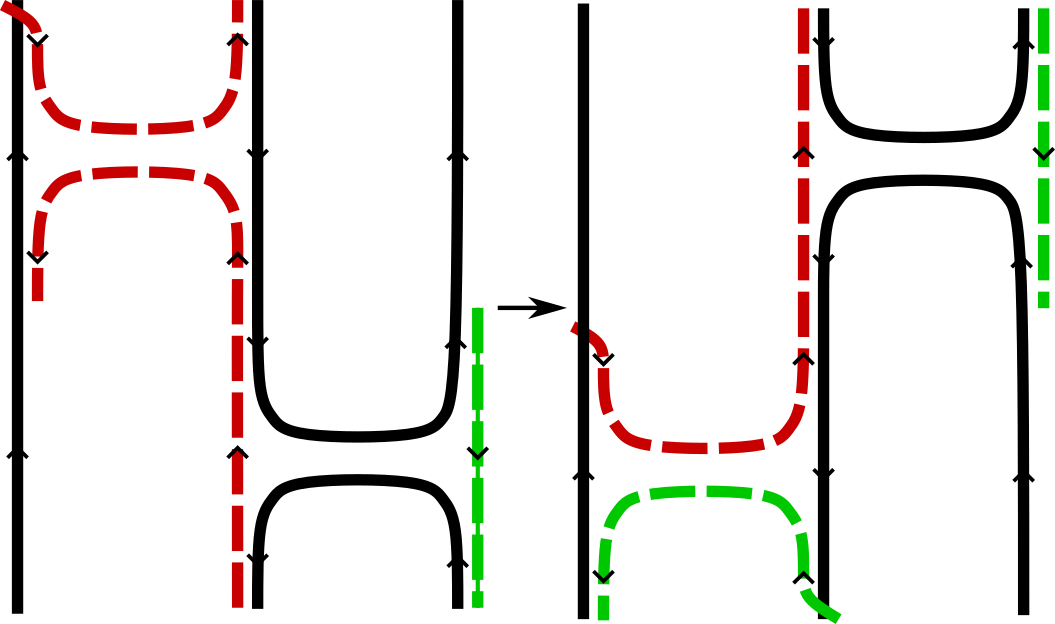}
		\label{rule3}
	\end{minipage}
	\caption{Examples of (A) TL relation, (B) Inter-commutation, and (C) Intra-commutation}
	\label{rules}
\end{figure}


\subsubsection{Deriving additional rewriting rules by substitution.}\label{sec:rewriting2}


Since DNA origami structure has no internal scaffold loops, rewriting rules similar to $(1)$--$(5)$ can 
be applied  to concatenations of generators, that is, products of $\alpha$'s and $\beta$'s. 
We extend these rules to more 
general 
substitution rules for 
our specific case
of generators $\alpha_i$ and $\beta_i$
by considering other $\gamma$'s, for instance $\gamma=\alpha\beta$.
The composition diagrams show that 
such substitution rules describe the 
DNA origami staples/scaffold structure in the way we proposed above (see Fig.~\ref{substitution1}), while these
new structures produce rules that cannot be derived from the listed ones in $(1)$--$(5)$.
Therefore we
consider
 rewriting rules for  concatenations 
of generators  $\alpha$'s and $\beta$'s.
Furthermore, we focus on concatenations of generators with the same or `neighboring' indexes because only for these generators the ends of the staples can connect. However, $\a_i$ and $\bt_j$ ($i\not = j$) can swap their places (by the inter-communication rule $(4)$) and the factor  $\a_i\bt_{i+1}$, for example,
of a word in $\Sigma_n$ can be substituted with a factor $\bt_{i+1}\a_i$.
Further, we observe that by setting $\gamma\in \{\a_i\bt_i, \bt_i\a_i\}$, the idempotency rule $(1)$ holds
as seen in Fig.~\ref{substitution1}. Therefore there are only four cases to consider:
$\gamma\in \{\a_i\bt_i,\bt_i\a_i,\a_i\bt_i\a_i,\bt_i\a_i\bt_i\}$ and check the plausibility of corresponding  DNA diagrams.  

First,  consider $\gamma \in \{\alpha\beta,\beta\alpha\}$, where $\gamma_i$ indicates  $\alpha_i\beta_i$. Then substituting $\gamma$ into rewriting rules (1), (2), and (3) gives us new rewriting rules (1a), (2a), and (3a). For example, (1a) consists of $\alpha_i\beta_i\alpha_i\beta_i \rightarrow \alpha_i\beta_i$ and $\beta_i\alpha_i\beta_i\alpha_i \rightarrow \beta_i\alpha_i$.
Note that a provisional rewriting rule (5a) could easily be obtained by the rewriting rule (5), so we do not
consider it as a new rule. We also do not add rewriting rule (4a) since it conflicts with the structure of the scaffold,
as shown in Fig.~\ref{wrong_rule}. Notice that the scaffold strand at the top left is connected to the second strand only on the left side of the figure, and on the right hand side of the figure it is connected from the strand three. Next we consider $\gamma \in \{\alpha\beta\alpha,\beta\alpha\beta\}$, which gives us rewriting rules (1b), (2b) and (3b). Similarly as before, rules 
(4b) or (5b) are not added, (4b) because of incompatible staple strands, and (5b) because it can be derived from (5). In addition, (1b) can also be derived from (1) and (1a), so it is not considered as a new rule.
In the end, we are left with 10 unique rewriting rules which we use to define the general rewriting rules and the monoids.


\begin{figure}
     \begin{minipage}[h]{.55\textwidth}
        \centering
	\begin{minipage}[h]{.45\textwidth}
		\centering
		\includegraphics[scale=.10]{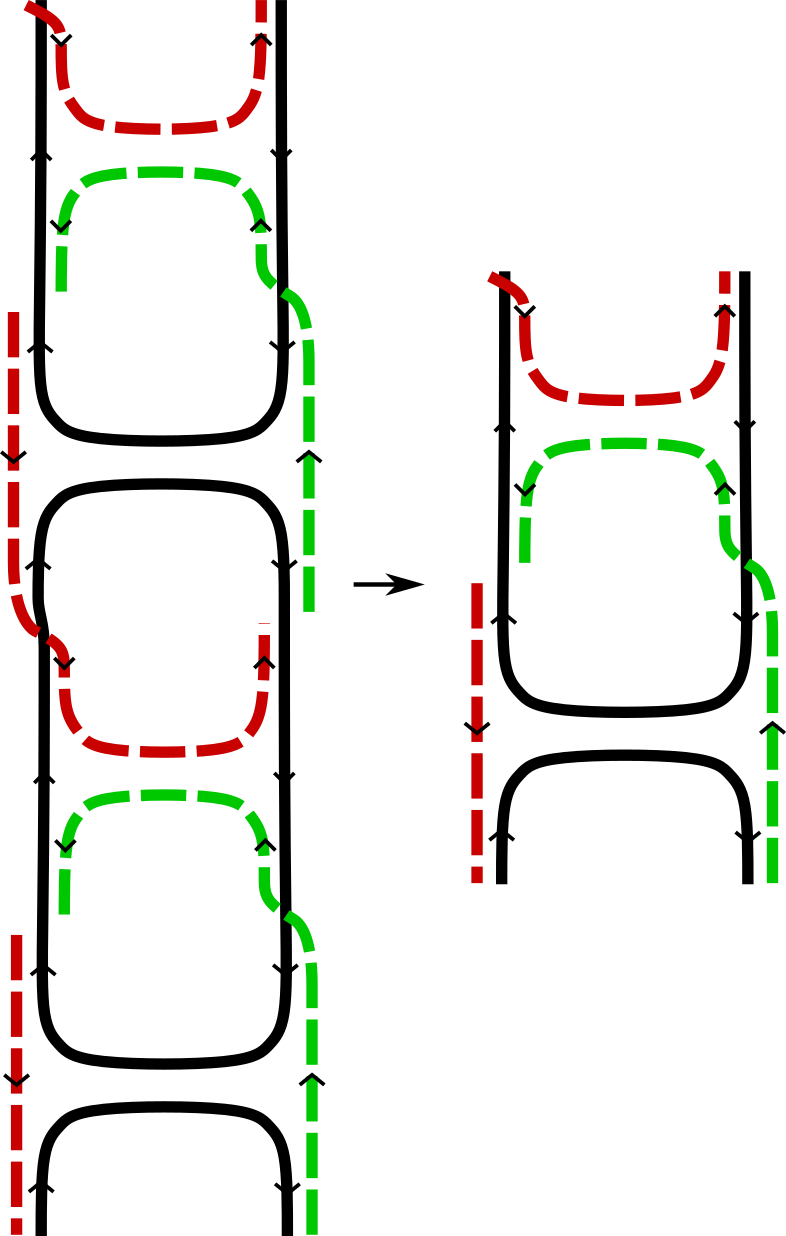}
		\label{abab}
	\end{minipage}%
	\begin{minipage}[h]{.45\textwidth}
		\centering
		\includegraphics[scale=.10]{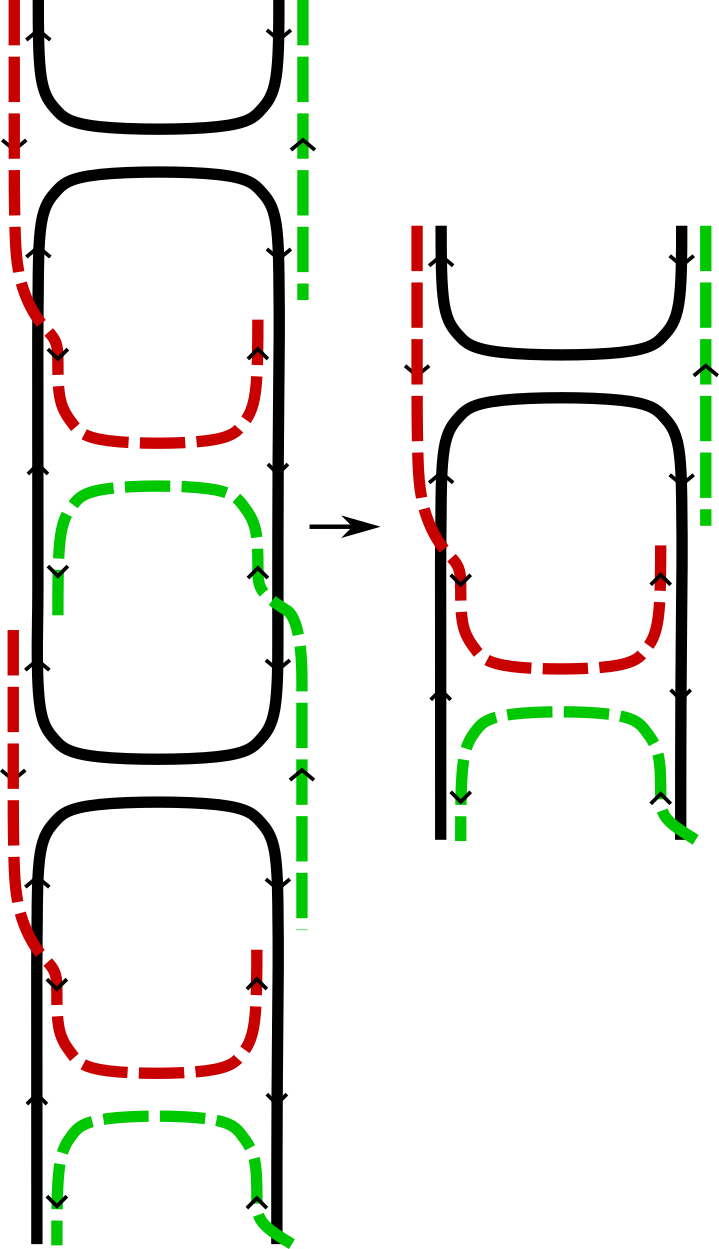}
		\label{baba}
	\end{minipage}
	\captionsetup{justification=centering}
	\caption{Substitution of $\alpha\beta$  and $\beta\alpha$ (resp.) into the first rewriting rule ($i$ odd)}
	\label{substitution1}
    \end{minipage}%
    \begin{minipage}[h]{.35\textwidth}
        \centering
	\captionsetup{justification=centering}
	\includegraphics[scale=.10]{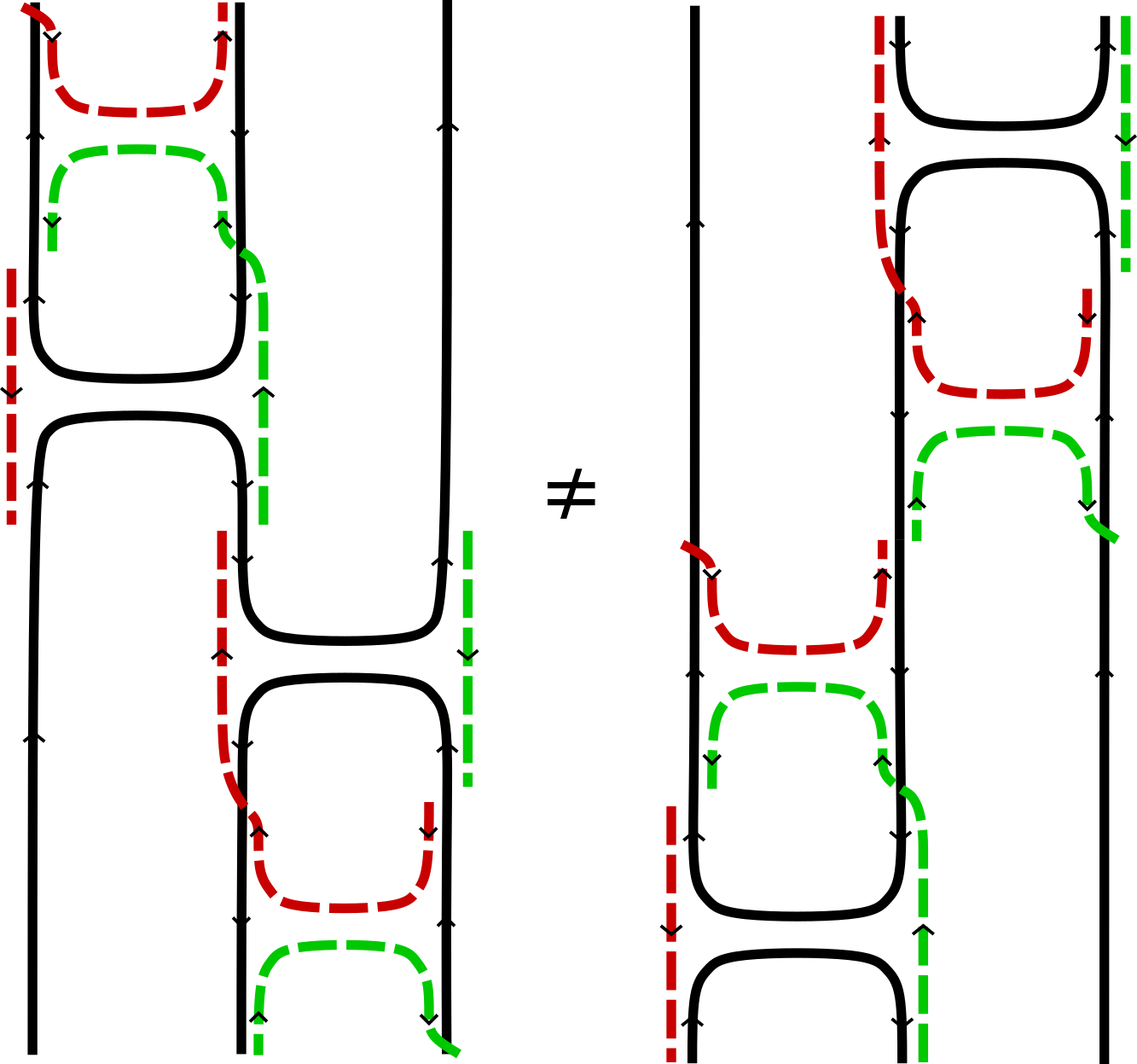}
	\caption{Substitution of $\gamma=\alpha\beta$ into rewriting rule (4) for $i$ odd
	}
	\label{wrong_rule}
    \end{minipage}
\end{figure}



\begin{definition}\label{origami-monoid}
	{\rm
		The \emph {origami monoid}  ${\mathcal O}_n$  is the monoid with a set of generators $\Sigma_n$  and relations generated by
		the rewriting rules (1) through (5), (1a), (2a), (3a), (2b), (3b).
	}
\end{definition}


\section{Monoid structures of ${\mathcal O}_n$ }
\label{sec:monoid2}


In this section, we present computational results on Green's $\mathscr{D}$-classes and compare them to 
those for the Jones monoids obtained in \cite{IdemP}. For comparison, we use the monoid epimorphism 
from ${\cal O}_n$ to the product ${\cal J}_n \times {\cal J}_n$ defined below. 

Let ${\mathcal J}_n$ be the Jones monoid of degree $n$ with generators $h_i$, $i=1, \ldots, n-1$.
We  denote the submonoid of $\mathcal{O}_n$ generated by $\alpha$s
(resp. $\beta$s),
by $\mathcal{O}^\alpha_n$ (resp. $\mathcal{O}^\beta_n$). An equivalent description for $\mathcal{O}^\alpha_n$ is the set of all words consisting of only $\alpha$s (plus the empty word), and similarly for $\mathcal{O}^\beta_n$. Let $\mathcal{O}^{\alpha \beta}_n= [\mathcal{O}_n \setminus (\mathcal{O}^\alpha_n \cup \mathcal{O}^\beta_n)] \cup \{1\}$.

\begin{lemma}\label{sub}
$\mathcal{O}^{\alpha \beta}_n$ is a submonoid of $\mathcal{O}_n$.
\end{lemma}

\begin{proof}
   The left and right hand sides of each rewriting rule show that rewriting a word by these rules does not change the absence, or existence of at least one $\alpha$ in the word, and similarly for $\beta$. Thus  multiplication of two words in $\mathcal{O}^{\alpha \beta}_n$ does not remove $\alpha$'s or $\beta$'s from the product, hence the product  remains in $\mathcal{O}^{\alpha \beta}_n$. 
\end{proof}

Let $p_\alpha : {\mathcal O}_n \rightarrow {\mathcal J}_n$
be the epimorphism defined by `projections' $p_\alpha (\alpha_i) = h_i$ and $p_\alpha (\beta_i) =1 $, 
for all $i=1, \ldots, n-1$, and let $p_\beta$ be defined similarly for $\beta$s. Define
$p: {\mathcal O}_n \rightarrow {\mathcal J}_n \times {\mathcal J}_n$
by
$ p(x) = (p_\alpha (x), p_\beta (x) ) $ for $x \in {\mathcal O}_n$.
Since the monoid relations of ${\mathcal O}_n$ hold under $p$, we have the following.

\begin{lemma}\label{iso}
    $\mathcal{O}_n^{\alpha} \cong \mathcal{O}_n^{\beta} \cong \mathcal{J}_n$.
\end{lemma}

\begin{lemma}
	The map $p: {\mathcal O}_n \rightarrow {\mathcal J}_n \times {\mathcal J}_n$ is a surjective monoid morphism.
\end{lemma}

In particular, it follows that the order of ${\cal O}_n$ is at least $|{\cal J}_n|^2$. 


\subsection{Orders of origami monoids}

For $n=2$ we can determine the order of ${\cal O}_2$ as follows.

\begin{lemma}\label{lemma:n=2}
	Every non-empty word of ${\mathcal O}_{2}$ can be reduced by rewriting rules to one of
	the following normal forms:
	$\alpha_{1},\
	\beta_{1},\
	\alpha_{1}\beta_{1},\
	\beta_{1}\alpha_{1},\
	\alpha_{1}\beta_{1}\alpha_{1},\ {\rm  or}\
	\beta_{1}\alpha_{1}\beta_{1}. $
\end{lemma}
\begin{proof}
	Since $\Sigma_{2}=\{\alpha_{1},\beta_{1}\}$, we list the words of length $3$ or less  exhaustively. After applying  rewriting rules to these words, they reduce to those words listed in the statement.
	
	Now consider a word $w$ with length greater than 3. We show that $w$ can be reduced to a word with length 3 or less.
	If $\alpha_{1}\alpha_{1}$ or $\beta_{1}\beta_{1}$ are factors of $w$, we
	reduce them to $\alpha_{1}$ or $\beta_{1}$, respectively.
	Repeating this process, we may assume that
	$w$ is an alternating sequence of $\alpha_{1}$ and $\beta_{1}$.
	Since $\alpha_{1}\beta_{1}$ and  $\beta_{1}\alpha_{1}$ are idempotent,
	$w$ reduces to a word of length less than 4.
\end{proof}

It is known that the  elements of the Jones monoid ${\mathcal J}_n$ of degree $n$ are in bijection with the linear chord diagrams obtained from the arcs of the diagrams representing them, and the total number  of such chord diagrams is equal to  the Catalan number
$\displaystyle C_n=\frac{1}{n+1}
\left( \begin{array}{cc} 2n \\ n \end{array} \right)$~\cite{BDP}.
Thus the numbers of  elements of ${\cal J}_n$ for $n=2, \ldots, 6$ are 
2, 5, 14, 42, 429, respectively.
GAP computations show that 
the number of non-identity elements in  ${\mathcal O}_3$, ${\mathcal O}_4$, ${\mathcal O}_5$ and $\mathcal O_6$ are
44, 293, 2179, 19086 respectively. This sequence of integers is not listed in the OEIS~\cite{OEIS} list of sequences. 
We observe that the orders of origami monoids are much larger. In fact it is not apparent from the definition whether they are all finite.
Thus we conjecture the following.

\begin{conjecture}\label{conj:finite}
	{\rm
		The order of ${\cal O}_n$ is finite for all $n$.
	}
\end{conjecture}

\subsection{Green's classes}



We have the following observations for 
Green's classes of $\mathcal{O}_n$ for general $n$.



\begin{figure}[h]
    \begin{minipage}[h]{.49\textwidth}
       
        \begin{minipage}[h]{.49\textwidth}
            \centering
		    \includegraphics[scale=.45]{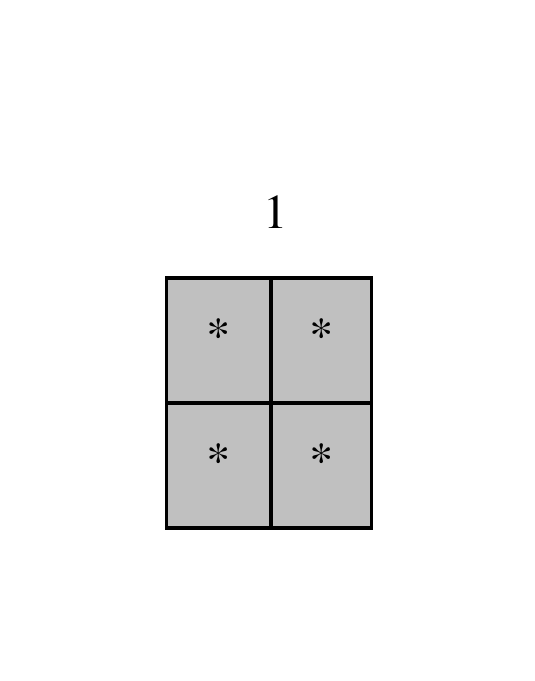}
        \end{minipage}
        \begin{minipage}[h]{.49\textwidth}
            \centering
		    \includegraphics[scale=.13]{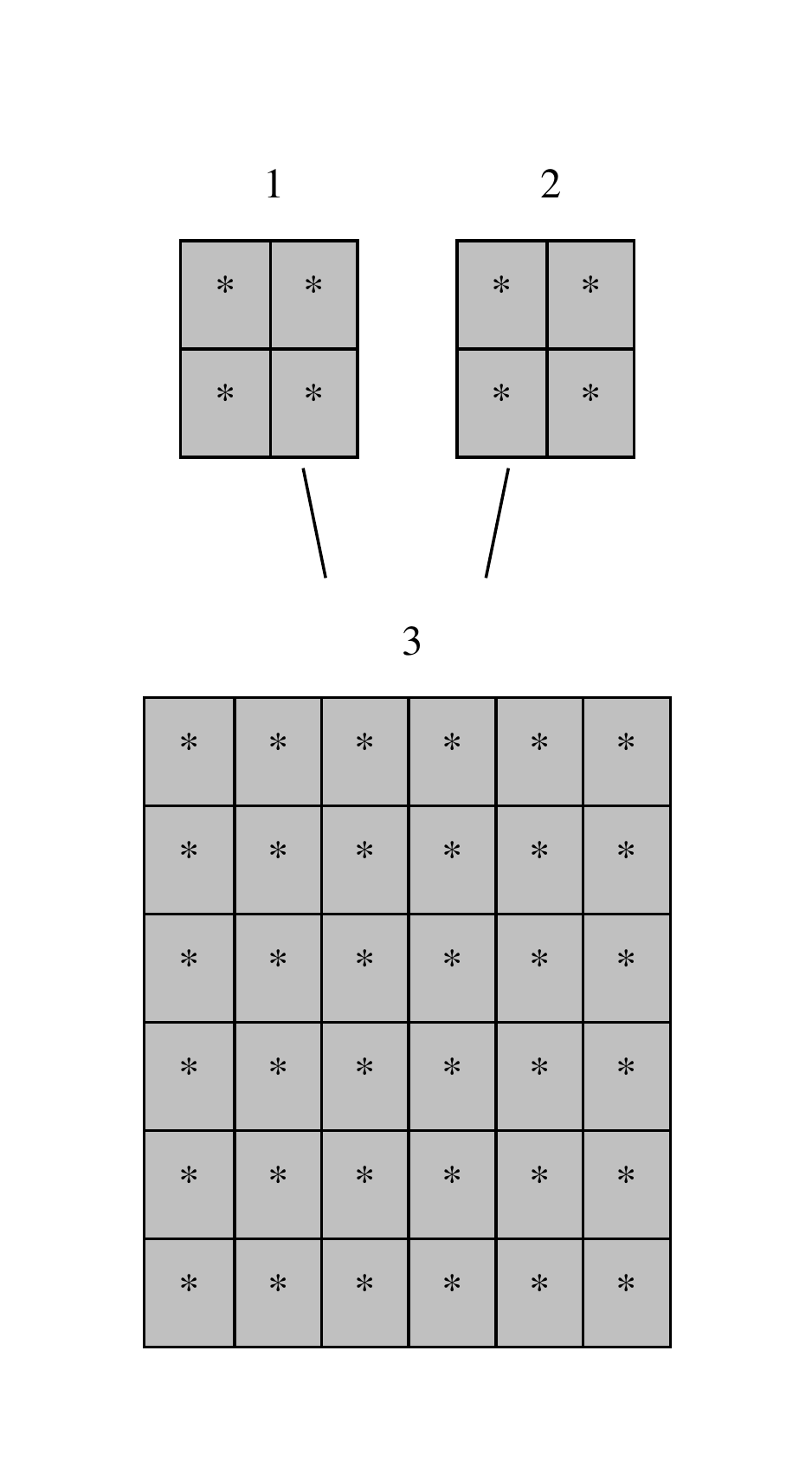}
        \end{minipage}
        
        \vspace{-.5 cm}
        
         \begin{minipage}[h]{.49\textwidth}
            \centering
		    \includegraphics[scale=.33]{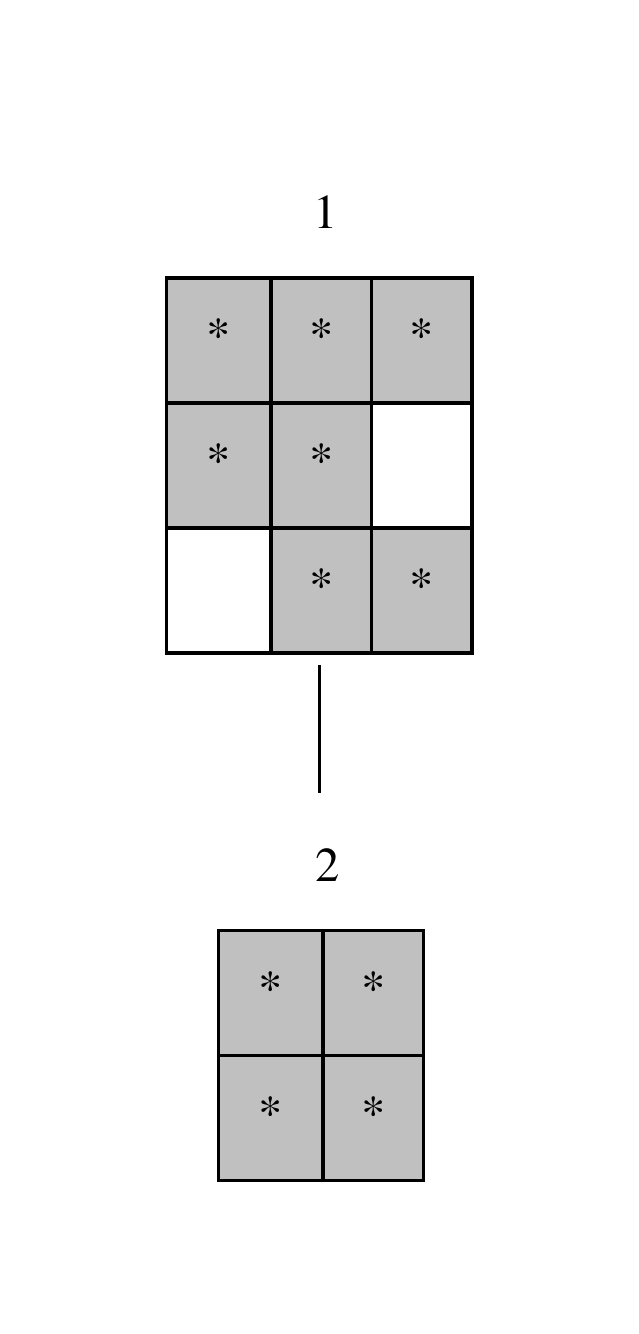}
        \end{minipage}
        \begin{minipage}[h]{.49\textwidth}
            \centering
		    \includegraphics[scale=.075]{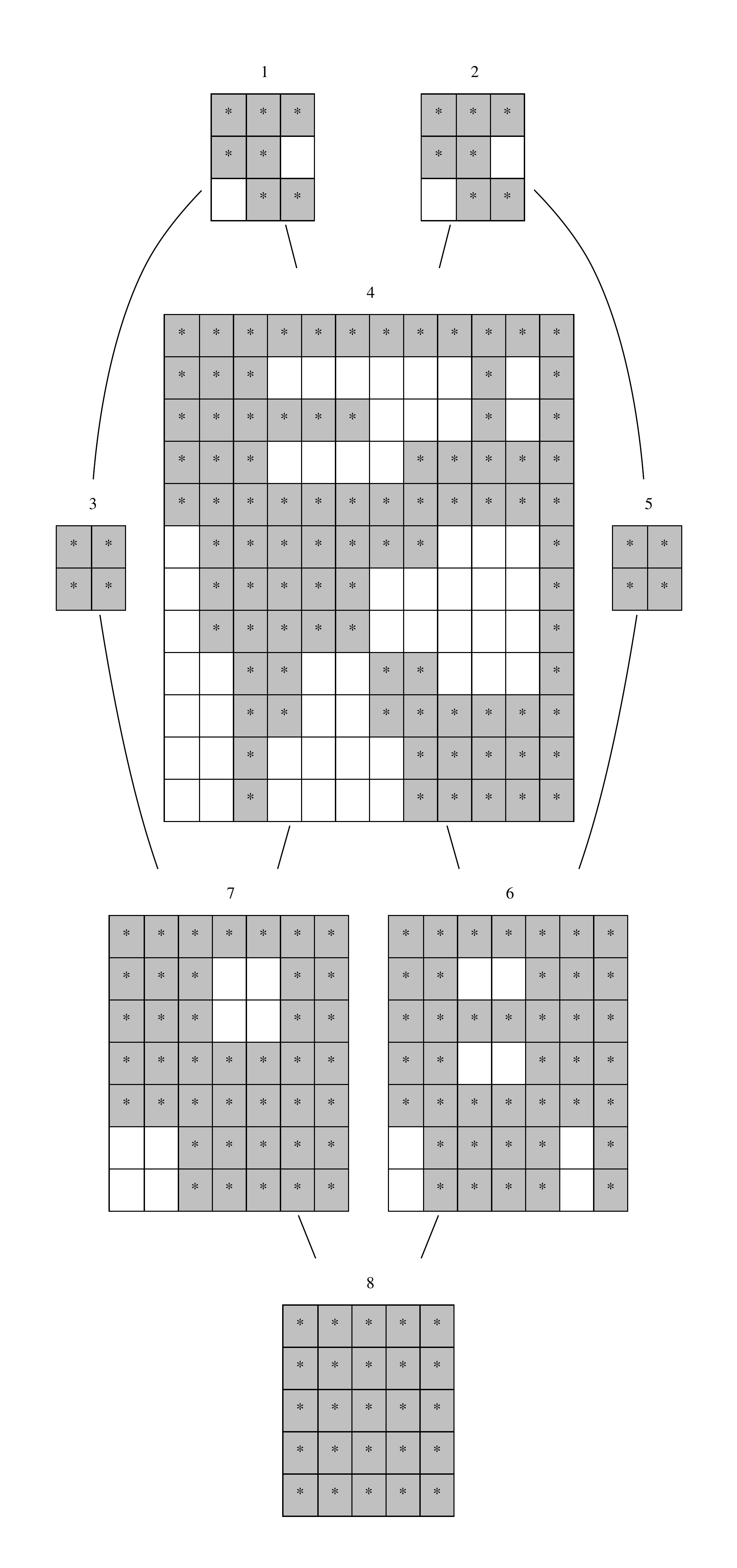}
        \end{minipage}
    \end{minipage}%
    \begin{minipage}[h]{.49\textwidth}
        \begin{minipage}[h]{.40\textwidth}
            
		    \includegraphics[scale=.30]{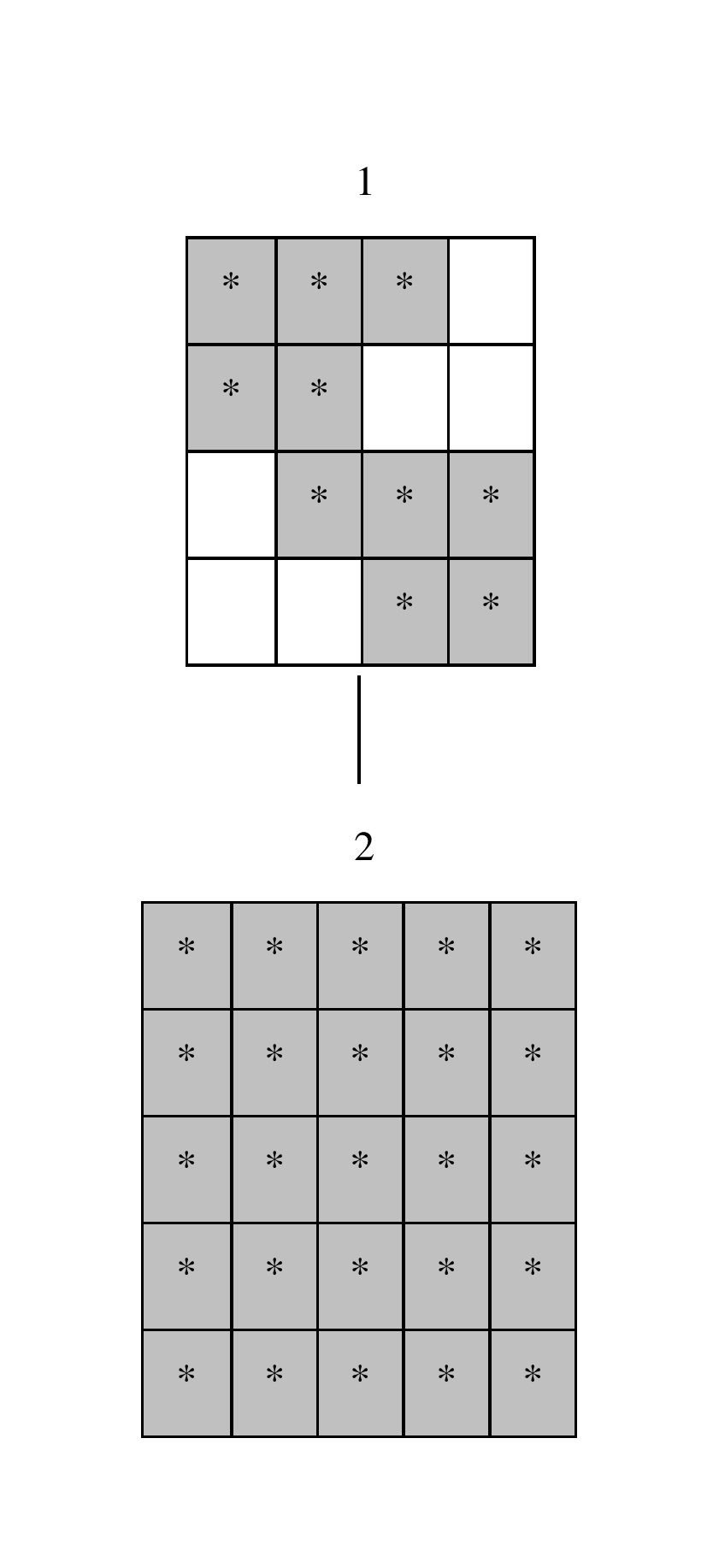}
        \end{minipage}
        \begin{minipage}[h]{.40\textwidth}
            
		    \includegraphics[scale=.06]{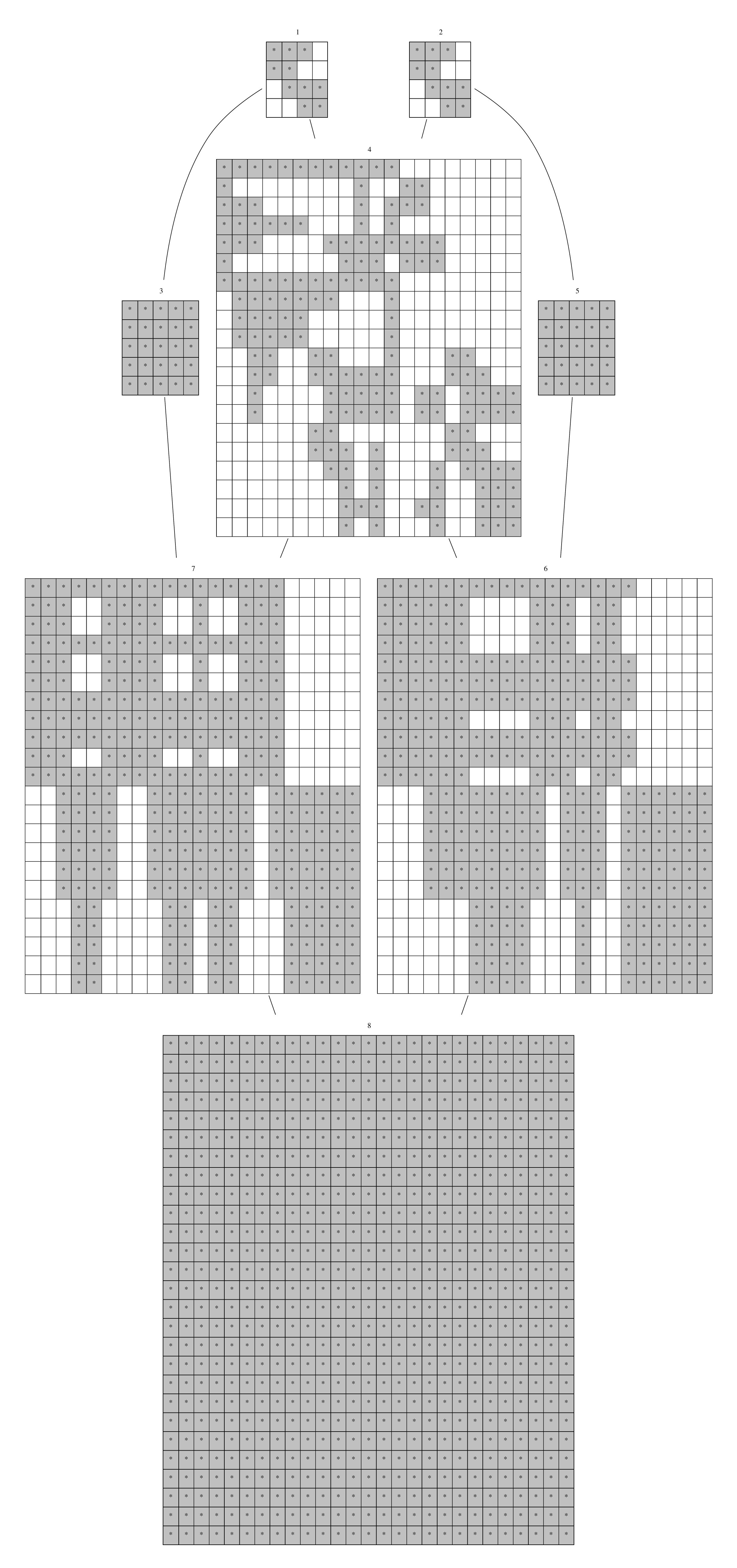}
        \end{minipage}
    \end{minipage}
    \caption{$\mathscr{D}$-classes of $\mathcal{J}_n$ (left) and $\mathcal{O}_n$ (right) for $n=3$ (top left), $n=4$ (bottom left), and  $n=5$ (right)}
    \label{Dclasses}
\end{figure}

\begin{lemma}\label{lemma1}
Let $x \in \mathcal{O}_n^{\alpha}$, $y \in \mathcal{O}_n^{\alpha\beta}$ be nonempty words and let $D_x$ and $D_y$ be the $\mathscr{D}$-classes containing $x$ and $y$, respectively. Then $D_x \neq D_y$.
\end{lemma}

\begin{proof}
   By Lemma \ref{sub}, if $y \mathscr{L} a$, then $a \in \mathcal{O}_n^{\alpha\beta} $, and if $a \mathscr{R} b$, then $b \in \mathcal{O}_n^{\alpha\beta}$. Thus we cannot have $y \mathscr{D} x$.
\end{proof}

\begin{corollary}
The conclusion of Lemma \ref{lemma1} holds for $x \in \mathcal{O}_n^{\beta}$, $y \in \mathcal{O}_n^{\alpha\beta}$ and $x \in \mathcal{O}_n^{\alpha}$, $y \in \mathcal{O}_n^{\beta}$.
\end{corollary}



\begin{remark}\label{remark:onto}
	If $\mc O_n$ is finite, then 
	each $\mathscr{D}$-class of $\mathcal{J}_n \times \mathcal{J}_n$ is an image of a $\mathscr{D}$-class of $\mathcal{O}_n$ by $p$.
	It follows from the definition of $p$ that every $D$-class of $\mc O_n$ maps into a $\mathscr{D}$-class of $\mc J_n$, and by Lemma 1.4 Ch. 5 in~\cite{Grillet} the map is also onto. We conjecture that this $\mathscr{D}$-class of $\mathcal{O}_n$ is unique. We show that this observation is true for $n\le 6$.
\end{remark}	

    

\subsection{Green's classes for $n\le 6$}
In this section we describe Green's $\mathscr{D}$-classes for $n \leq 6$. 
We used {\it GAP} to determine $\mathscr{D}$-classes of origami monoids ${\cal O}_n$ for $n \leq 6$, the structure is presented in 
Figs.~\ref{Dclasses} and~\ref{fig:d6}.
Shaded squares represent $\mathscr{H}$-classes which contain an idempotent. We note that for $n \leq 6$, every $\mathscr{H}$-class of $\mathcal{O}_n$ is singleton, so each square in the figure represents precisely one element of $\mathcal{O}_n$.

For $n \leq 6$, since $\mathcal{O}_n$ is finite , the $\mathscr{J}$ and $\mathscr{D}$ relations coincide. A preorder $\leq_{\mathscr{D}}$ is defined on $\mathcal{O}_n$ by $a \leq_{\mathscr{D}} b$ if the two-sided principal ideal generated by $a$ is a subset of the two-sided principal ideal generated by $b$. This condition is equivalent to the existence of $x,y \in \mathcal{O}_n$ such that $xby=a$. Since any two elements of a $\mathscr{D}$-class generated the same principal ideal, this preorder may be extended to the set of $\mathscr{D}$-classes of $\mathcal{O}_n$ such that $D \leq_{\mathscr{D}} D'$ if for $a \in D$ and $b \in D'$, $a \leq_{\mathscr{D}} b$. The lines between $\mathscr{D}$-classes in the figures represent the lattice structure of this preorder. 

The relations between $\mathcal{O}_n$ and $\mathcal{J}_n$ described in Section 6.2 can be observed in Fig.~\ref{Dclasses}. We omit the $\mathscr{D}$-class consisting of only the empty word from the diagrams, which is maximal in the lattice of $\mathscr{D}$-classes. For each $n$, 2 copies of the $\mathscr{D}$-classes of $\mathcal{J}_n$ can be found as the $\mathscr{D}$-classes of $\mathcal{O}_n^{\alpha}$ and $\mathcal{O}_n^{\beta}$, respectively, in the $\mathscr{D}$-classes of $\mathcal{O}_n$. As described in Remark \ref{remark:onto}, these correspond to the cross product of one identity and one non-identity $\mathscr{D}$-class of $\mathcal{J}_n$. The other $\mathscr{D}$-classes are those of $\mathcal{O}_n^{\alpha\beta}$, and correspond to the cross product of two non-identity $\mathscr{D}$-classes of $\mathcal{J}_n$. Which pair of $\mathscr{D}$-classes of $\mathcal{J}_n$ correspond to which $\mathscr{D}$-class of $\mathcal{O}_n$ can be better seen in Fig. \ref{fig:d6}.

In Fig.~\ref{fig:d6}, we arrange the $\mathscr{D}$-classes of $\mathcal{O}_6$ to better illustrate the relation between the $\mathscr{D}$-classes of $\mathcal{J}_n$, although the same process may be applied to other $n$. On the right, the lattice structure of the $\mathscr{D}$-classes remains, applying left-to-right as well as top-to-bottom. The $\mathscr{D}$-classes along the top row and left column are the $\mathscr{D}$-classes of $\mathcal{O}_6^{\alpha}$ and $\mathcal{O}_6^{\beta}$ respectively, as previously described isomorphic to $\mc J_n$. For any $\mathscr{D}$-class of $\mathcal{O}_n^{\alpha\beta}$, the $\mathscr{D}$-classes which it maps onto are greater in the lattice structure. Thus the grid of $\mathscr{D}$-classes may be thought of as a table, with the row and column of any entry determining the image of the $\mathscr{D}$-class by $p_{\alpha}$ and $p_{\beta}$, respectively. Since rewriting relations are equivalent for $\alpha$ and $\beta$, the $\mathscr{D}$-classes are symmetric with respect to switching rows and columns. This can be easily seen in the $\mathscr{D}$-classes in the upper right and lower left corners. However, the rows and columns of any $\mathscr{D}$-class may be ordered arbitrarily, and are automated by GAP, making the symmetry non-obvious for other $\mathscr{D}$-classes.

\begin{figure}
    \begin{minipage}[h]{.49\textwidth}
        \includegraphics[scale=.02]{images/6stringsDClasses}
    \end{minipage}%
    \begin{minipage}[h]{.49\textwidth}
        \includegraphics[scale=.25]{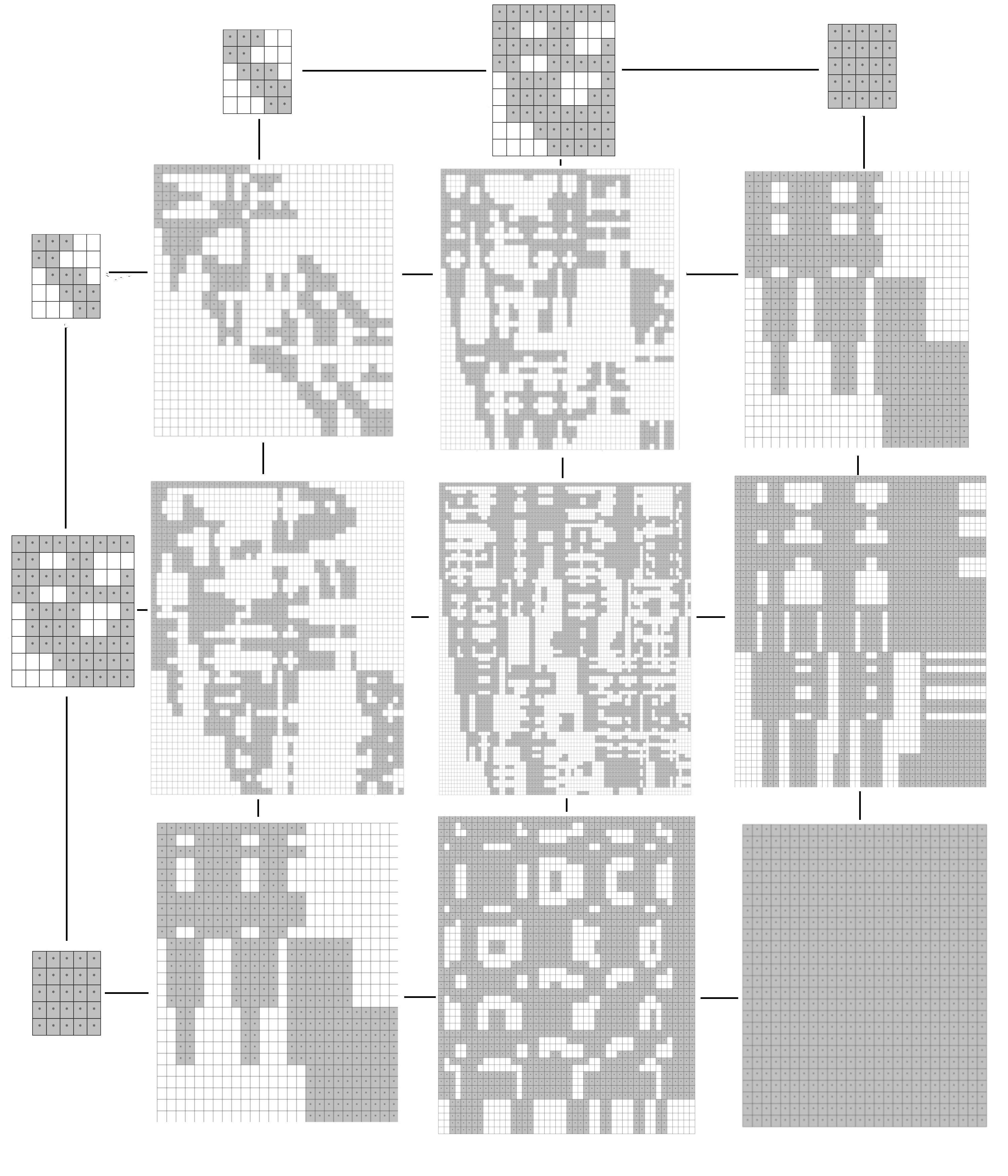}
    \end{minipage}

    \caption{$\mathscr{D}$-classes of $\mathcal{O}_6$ (left) and  re-arranged
    and resized to fit the grid (right)}
    \label{fig:d6}
\end{figure}

\vspace{-.5cm}

\section{Concluding remarks}
In this paper, motivated from similarity to Temperley-Lieb algebras,  we  introduced an algebraic system that describes DNA origami structures.
Generators in this system are defined such that they mimic basic building blocks of DNA origami. Following the structural properties of the DNA origami, we established rewriting rules as well as monoids whose elements conform to the relations obtained from these rules. To each DNA origami structure we can associate an element from an appropriate monoid. For example, the structure in Fig.~\ref{origami_clean} corresponds to the element whose normal form is $\alpha_1\alpha_3\alpha_5\beta_2\beta_4$.
We hope that such representations of DNA origami may provide a tool for distinguishing constructs.

The monoids introduced here are generalizations of Temperley-Lieb algebras, and we provide several conjectures with the goal of relating them to known monoids. For example, from our findings for $n\leq 6$, we conjecture that ${\mathcal O}_n$ are finite for all $n$, and $\mathscr{H}$-classes are singletons.
We also provide conjectures relating to the $\mathscr{D}$-classes of $\mathcal{O}_n$ and $\mathcal{J}_n$ under the morphism $p$.
Specifically, we conjecture that the $\mathscr{D}$-classes of $\mathcal{O}_n$ are in one-to-one correspondence with the $\mathscr{D}$-classes of $\mathcal{J}_n \times \mathcal{J}_n$.


\section*{Acknowledgment}
This work is partially supported by 
NIH R01GM109459, and by 
NSF's CCF-1526485 and DMS-1800443.
This research was also partially supported by the Southeast Center for Mathematics and Biology, an NSF-Simons Research Center for Mathematics of Complex Biological Systems, under National Science Foundation Grant No. DMS-1764406 and Simons Foundation Grant No. 594594.

\end{document}